\long\def\symbolfootnote[#1]#2{\begingroup\def\thefootnote{\fnsymbol{footnote}}\footnote[#1]{#2}\endgroup}

\documentclass{amsart}

\usepackage{hyperref}
\usepackage{indentfirst}
\usepackage{amssymb}
\usepackage{amsmath}
\usepackage{amsfonts}

\newtheorem{theorem}{Theorem}[section]

\newtheorem{lemma}[theorem]{Lemma}
\theoremstyle{remark}

\theoremstyle{definition}
\newtheorem{definition}[theorem]{Definition}

\theoremstyle{proposition}
\newtheorem{proposition}[theorem]{Proposition}
\newtheorem{conjecture}[theorem]{Conjecture}
\numberwithin{equation}{section}

\begin{document}
\author{Linfeng Zhou}
\title[]{The isoperimetric problem in the 2-dimensional Finsler space forms with $k=0$. \uppercase\expandafter{\romannumeral1}}
\date{}
\maketitle

\begin{abstract} 
In this paper, the isoperimetric problem in the 2-dimensional Finsler space form $(F_B, B^2(1))$ with $k=0$ by using the Busemann-Hausdorff area is investigated.  We prove that the circle centered the origin achieves the local maximum area of the isoperimetric problem.\\

\end{abstract}

\section{\textbf{Introduction}}

In Riemannian geometry, a Riemann space form usually defined as a simple-connected Riemannian manifold $(M^n,g)$ with constant sectional curvature $k$.  It is well-known that the Riemann space forms are isometric to either the sphere $S^n$, or the Euclidean space $R^n$, or the hyperbolic space $H^n$.  These Riemann space forms are homogenous, locally projective flat and give the good models in Riemannian geometry.   

In Finsler geometry, what is the meaning of a Finsler space form? One way is to define the Finsler manifold with constant flag curvature to be a Finsler space form. However, till now we do not have a complete classification theorem of the Finsler metric of constant flag curvature. Even to some special class: such as the $(\alpha, \beta)$ metrics, how to classify these metrics with constant flag curvature is an open problem. 

Here is another way to define a Finsler space form. 
In a local coordinate, if imposing a spherical symmetry on a Finsler metric $(\Omega, F)$, then $F$ can be expressed by 
\[F(x,y)=|y|\phi(|x|,\frac{\langle x,y\rangle}{|y|}).\]
When considering the spherically symmetric projectively flat Finsler metric with constant flag curvature, the following classification theorem (see \cite{Zh}, \cite{MZ} and \cite{L} )is proved :

 \begin{theorem}Up to a scaling, a spherically symmetric projectively flat Finsler metric $(\Omega, F)$ has a constant flag curvature $k$ if and only if 
\begin{enumerate}
\item[(1)] $k=1$, the metric is either the projective sphere model or Bryant metric \[F=\frac{|y|c_1(z_1)}{c_1(z_1)^2+\big(z_2+c_2(z_1)\big)^2}\] where
$z_1:=\sqrt{|x|^2-\frac{\langle x,y\rangle^2}{|y|^2}}, z_2:=\frac{\langle x,y\rangle}{|y|},
c_1(z_1):=\frac{\sqrt{2}}{2}\sqrt{2d_2+z_1^2+\sqrt{(2d_2+z_1^2)^2+4d_1^2}},$
$$c_2(z_1):=\pm\frac{\sqrt{2}}{2}\sqrt{-2d_2-z_1^2+\sqrt{(2d_2+z_1^2)^2+4d_1^2}},$$
$d_1$ and $d_2$ are positive real numbers, $\Omega=R^n$.

\item[(2)] $k=0$, the metric is either the Euclidean space $R^n$ or the Berwald's example where 
\[F=\frac{(\sqrt{|y|^2-(|x|^2|y|^2-\langle x,y\rangle^2)}+\langle x,y\rangle)^2}{(1-|x|^2)^2\sqrt{|y|^2-(|x|^2|y|^2-\langle x,y\rangle^2)}}, \Omega=B^n(1).\]

\item[(3)] $k=-1$, the metric is either the Klein metric (one model of the hyperbolic space)
\[F=\frac{\sqrt{|y|^2-(|x|^2|y|^2-\langle x,y\rangle^2)}}{1-|x|^2}, \Omega=B^n(1); \]
or the Funk metric of Randers type
\[F=\frac{\sqrt{|y|^2-(|x|^2|y|^2-\langle x,y\rangle^2)}+\langle x,y\rangle}{2(1-|x|^2)}, \Omega=B^n(1); \]
or $\Omega=B^n(\sqrt{2(d_2-d_1)})$ with Finsler metric
\[F=\frac{|y|c_1(z_1)}{c_1(z_1)^2-\big(z_2+c_2(z_1)\big)^2}\] where
$z_1:=\sqrt{|x|^2-\frac{\langle x,y\rangle^2}{|y|^2}}, z_2:=\frac{\langle x,y\rangle}{|y|},
c_1(z_1):=\frac{\sqrt{2}}{2}\sqrt{2d_2-z_1^2+\sqrt{(2d_2-z_1^2)^2-4d_1^2}},$
\[c_2(z_1):=\pm\frac{\sqrt{2}}{2}\sqrt{2d_2-z_1^2-\sqrt{(2d_2-z_1^2)^2-4d_1^2}},\] and $d_2>d_1$ are positive real numbers.
\end{enumerate}
\end{theorem}
In this theorem,  the Riemannian space forms are also included. Furthermore, the non-Riemannian metrics share the same geodesics of the Remannian space forms and have a nice symmetry. It leads us to introduce the following definition. 
\begin{definition}
The metrics in the above theorem are called the \textbf{Finsler space forms}. 
\end{definition}
Till now, the geometry of the Riemannian space  forms is well-known: such as the geodesics, the trigonometry, the volume of the geodesic ball, all kinds of comparison theorem etc. However, the geometry of the non-Riemannian Finsler space forms are not  fully comprehended. One aspect of the geometry in the 2-dimensional Riemannian space forms is the isoperimetric problem: maximize the area of a simple closed curve under the constraint that the length of the curve to be fixed. 

For 2-dimensional Euclidean space, this problem is quite famous and old \cite{Bl}. It can be summarized as the following isoperimetric inequality 
\[L^2\geq 4\pi A,\]
where $L$ is the length of the closed curve and $A$ is the area enclosed by the curve. The inequality holds if and only if the curve is a circle in the plane. Although the solution to the isoperimetric problem was known in Greece. But the first rigorous proof was obtained only in the 19th century. 
In 1838, Steiner \cite{St} gave an elegant and simple proof of a condition necessary for a solution but did not give a sufficiency proof. The first complete proof was given by Weierstrass via the calculus of variations (see \cite{Bl}). In 1902, Hurwitz \cite{Hu} published a short proof using the Fourier series. An elegant direct proof was given by E. Schmidt \cite{Sc} in 1938 using the Cauchy-Schwarz inequality.  
     
As to the 2-dimensional sphere and the hyperbolic plane, the isoperimetric problem is given by 
\[L^2\geq 4\pi A-kA^2,\]
where $L$ is the length of the closed curve, $A$ is the area enclosed by the curve  and $k$ is the curvature of the space (see \cite{Ch}). The equality holds if and only if the curve is a geodesic circle. Above isoperimetric inequality can be generalized to the higher dimensional Riemannian space forms.

In Finslerian case, H. Busemann \cite{Bu} gave the solution of the isoperimetric problem in Minkowski space and proved the inequality
\[\int_{D}\sigma(u)dS\geq n V^{(n-1)/n}(D)V^{1/n}(K)\]
where $K$ is the polar reciprocal (with respect to the unit sphere) to the boundary of the convex closure of the surface $C: \sigma^{-1}(u)u$; and the equality sign holds only 
for $D$ homothetic to $K$.

The purpose of this paper is to consider the isoperimetric problem in the 2-dimensional Finsler space form $(F_B, B^2(1))$ with $k=0$ by using the Busemann-Hausdorff area.  The author employs the method of the calculus of variations and proves that \emph {the circle centered the origin archives the local maximum area, which is called a proper maximum, of the isoperimetric problem.} Actually, we conjecture that the circle centered the origin is just the solution of the isoperimetric problem in $(F_B, B^2(1))$.

Section 2 is a quick introduction of the theory of the calculus of variations which the author will use in the procceding sections. In section 3, a formula of the Busemann-Hausdorff area of $(F_B, B^2(1))$ is given. Section 4 discusses the Euler-Lagrange equation of the isoperimetric problems and obtains that the circle centered the origin is isoperimetric extremal. In section 5 and 6, it is proved that the Weierstrass function is negative, and there is no conjugate point along the circle centered the origin. The final section gives a complete proof of the main result based on the previous sections.

\section{\textbf{A brief review of the sufficiency theorem for isoperimetric problem in the calculus of variations}}
 
In the classical calculus of variations, the isoperimetric problem to be considered is that of minimizing an integral 
$$I=\int_{a}^{b}f(t,x(t),\dot{x}(t))dt=\int_{a}^{b}f(t,x_1,\dots,x_n,\dot{x}_1,\dots,\dot{x}_n)dt$$
in a class of admissible arcs
$$x_i(t), a\leq t\leq b; i=1,\dots,n,$$
joining two fixed points and satisfying a set of isoperimetric conditions 
$$I_{\alpha}=\int_{a}^{b}f_{\alpha}(t,x(t),\dot{x}(t))dt=l_{\alpha}, \alpha=1,\dots,m,$$
where the $l's$ are constants. 
It is assumed that the functions $f, f_{\alpha}$ are defined and have continuous derivatives of the first three orders in a region $R$ of points $(t,x,\dot{x})$. The points of $R$ will be called admissible. A continuous arc that can be divided into a finite number of suburbs on each of which it has continuous derivatives will be called admissible if its elements $(t,x,\dot{x})$ are all admissible. 

Before stating the sufficiency theorem proved in \cite{He}, let us recall some concepts.

An admissible arc $x_0$ is called a strong minimum of above isoperimetric problem  if there is a neighborhood $\mathcal{F}$ of $x_0$ in $tx$-space such that the inequality $I(x)\geq I(x_0)$ holds for all admissible arcs $x\neq x_0$ whose elements $(t,x(t))$ are in $\mathcal {F}$. If the inequality can be replaced by a strict inequality, the minimum is said to be a proper strong minimum.

The sufficiency conditions are given in terms of an integral $J(x)$ of the form
$$J(x)=I+\sum_{\alpha=1}^{m}\lambda_{\alpha}I_{\alpha}=\int_{a}^{b}F(t,x(t),\dot{x}(t),\lambda)dt,$$
where $F=f+\sum_{\alpha=1}^{m}\lambda_{\alpha}f_{\alpha}$ and the $\lambda$'s are constant multipliers.  An admissible arc and a set of constants $\lambda_{\alpha}$ having continuous second derivatives will be said to form an isoperimetric extremal if they satisfy the Euler-Lagrange equations 
\begin{equation}
F_{x_i}-\frac{dF_{\dot{x}_i}}{dt}=0.
\end{equation}

Let an admissible $x_0$ be an isoperimetric extremal joining the two points and satisfying the condition (2). The arc $x_0$ is said to be normal, if the equation $P_{i\alpha}a_{\alpha}=0$ where 
$$P_{i\alpha}=(f_{\alpha})_{ x_i}-\frac{d (f_{\alpha})_{\dot{x}_i}}{dt},$$
hold along $x_0$ only in case the constants $a_{\alpha}$ are all zero. 

The Weierstrass E-function for $F$ is given by
$$E(t,x,\dot{x},u):=F(t,x,u)-F(t,x,\dot{x})-\sum_{i=1}^{n}(u_i-\dot{x}_i)F_{\dot{x}_i}(t,x,\dot{x}).$$
An arc $x_0$ is said to satisfy the strict Weierstrass condition if for each $(t,x,\dot{x})$ in a neighborhood of $x_0$, the Weierstrass function
$$E(t,x,\dot{x},u)>0$$ 
holds for every admissible set $(t,x,u)\neq (t,x,\dot{x})$.

Finally the second variation of $J^{\prime\prime}(x_0,y)$ of $J$ along $x_0$ is of the form 
$$J^{\prime\prime}(x_0,y)=\int_a^b2\omega(t,y(t),\dot{y}(t))dt,$$
where $2\omega=\sum_{i,j}F_{x_ix_j}y_iy_j+2F_{x_i\dot{x}_j}y_i\dot{y}_j+F_{\dot{x}_i\dot{x}_j}\dot{y}_i\dot{y}_j$.

A sufficiency theorem for a strong minimum of the isoperimetric problem, which is proved by Hestenes in \cite{He}, can be stated as the following:
 \begin{theorem}
 Let $x_0$ be an admissible arc. Suppose there exist $\lambda_1,\dots,\lambda_m$ such that, relative to the function
$$J(x)=I+\sum_{\alpha=1}^{m}\lambda_{\alpha}I_{\alpha}=\int_{a}^{b}F(t,x(t),\dot{x}(t),\lambda)dt,$$
\begin{enumerate}
\item $x_0$ is isoperimetric extremal,
\item $x_0$ is normal,
\item $x_0$ satisfies the strict Weierstrass condition,
\item $J^{\prime\prime}(x_0,y)>0$ for every non-null admissible variations $y_i(t), (a\leq t\leq b)$, vanishing at $t=a$ and $t=b$ and satisfying with $x_0$ the equations 
         $$\int_{a}^b\big( (f_{\alpha})_{ x_i}y_i+(f_{\alpha})_{ \dot{x}_i}\dot{y}_i \big)dt=0,$$
\item along $x_0$ the inequality $\sum_{i,j}F_{\dot{x}_i\dot{x}_j}y_iy_j>0$ holds for every set $(y)\neq (0).$        
\end{enumerate}
Then $x_0$ is a proper strong minimum of the isoperimetric problem. 
 \end{theorem}
 
 A point $t=c$ on $a<t\leq b$ will be said to define a conjugate point to $t=a$ on $x_0$ if there is a solution $y_i(t)$ of the Jacobi equations
 $$\omega_{y_i}-\frac{d\omega_{\dot{y}_i}}{dt}=0$$
such that 
$$\int_{a}^b\big( (f_{\alpha})_{ x_i}y_i+(f_{\alpha})_{ \dot{x}_i}\dot{y}_i \big)dt=0.$$
Furthermore, the solution $y(t)$  satisfies $y(a)=y(c)=0$ and $y(t)\not\equiv 0$ on $a<t<c$.

\begin{theorem}
The inequality $J^{\prime\prime}(x_0,y)>0$ holds for all $y\neq 0$ if and only if there is no point $c$ conjugate to $a$ on $0<t\neq b$.
\end{theorem}
  
For our purpose, we shall concern the calculus of variations in parametric form in 2-dimensional Euclidean space.  Let an arc be of the parametric form
  $$x(t):= (x_1(t), x_2(t))  \quad (a\leq t\leq b)$$
in $R^2$, where the functions $x_i(t)$ are assumed to be continuous first and second derivatives $\dot{x}_i(t)$ for $i=1, 2$ and if more over $\dot{x}_1^2+\dot{x}_2^2\neq 0$ in $(a,b)$. 

An integral 
$$\int_a^b f(t,x(t),\dot{x}(t))dt$$
can be shown to be independent of the parametrization of the arc $x$ if and only if the integrand $f$ is independent of $t$ and is positively homogeneous in $\dot{x}$ of degree one. 

The isoperimetric problem in parametric form can be stated as the following type.
Among all admissible curves joining two given points for which the integral
$$K=\int_a^b g(x_1,x_2,\dot{x}_1,\dot{x}_2)dt$$
takes a given value $l$, how to determine the one which minimizes the integral 
  $$I=\int_a^b f(x(t),\dot{x}(t))dt=\int_a^b f(x_1,x_2,\dot{x}_1,\dot{x}_2)dt.$$ 
The two function $f$ and $g$ are positively homogeneous in $\dot{x}$ of degree one and have continuous derivatives of the first three orders.

Let 
$$J(x)=I+\lambda K=\int_a^b H(x_1,x_2,\dot{x}_1,\dot{x}_2,\lambda)dt$$
where $H=f+\lambda g$. With a slight difference, the sufficient theorem proved by Hestenes can be illustrated by the following.
\begin{theorem} Let $x_0$ be an admissible arc. Suppose there exist $\lambda_0$ such that, relative to the function
$$J(x)=\int_{a}^{b}H(x_1,x_2,\dot{x}_1,\dot{x}_2,\lambda)dt,$$
\begin{enumerate}
\item $x_0$ is isoperimetric extremal, i.e.
  $$H_{x_i}-\frac{dH_{\dot{x}_i}}{dt}=0$$
  for $i=1,2$,
\item $x_0$ is normal which means the equation $P_{i}\neq 0$ where 
$$P_{i}=g_{ x_i}-\frac{d g_{\dot{x}_i}}{dt},$$
hold along $x_0$ for $i=1,2$,

\item  the Weierstrass function
$$E(x,\dot{x},u)>0$$ 
holds for every admissible set $(x,u)$ with  $u\neq k\dot{x} (k>0)$,

\item $J^{\prime\prime}(x_0,y)>0$ for every admissible variations $y_i(t)\neq \rho(t)\dot{x}_{0i}(t), (a\leq t\leq b)$, vanishing at $t=a$ and $t=b$ and satisfying with $x_0$ the equations 
         $$\int_{a}^b\big( g_{ x_i}y_i+g_{ \dot{x}_i}\dot{y}_i \big)dt=0\quad (i=1,2),$$
\item along $x_0$ the inequality $\sum_{i,j=1}^2H_{\dot{x}_i\dot{x}_j}y_iy_j>0$ holds for all $y\neq k\dot{x}_0(t).$        
\end{enumerate}
Then $x_0$ is a proper strong minimum of the isoperimetric problem. 
\end{theorem}

In the Bolza's book \cite{Bo}, the Jacobi equation along an isoperimetric extremal admissible curve $x_0$ turns out to be  
\[\Psi(y)+\mu U=0\]
where $\Psi(y)=H_2y-\frac{d}{dt}(H_1\dot{y})$,   $H_1=\frac{H_{\dot{x}_1\dot{x}_1}}{\dot{x}_2^2}$, 
$H_2=\frac{H_{x_1x_1}-\ddot{x}_2^2H_1-\frac{d L}{dt}}{\dot{x}_2^2}$, $L=H_{x_1\dot{x}_1}-\dot{x}_2\ddot{x}_2H_1$, $U=g_{x_1\dot{x}_2}-g_{\dot{x}_1x_2}+g_1(\dot{x}_1\ddot{x}_2-\ddot{x}_1\dot{x}_2)$, $g_1=\frac{g_{\dot{x}_1\dot{x}_1}}{\dot{x}_2^2}$. Moreover, the solution $y(t)$ shall satisfy the condition 
\[\int_a^b Uy dt=0.\]
 A conjugate point  $t=c$ to the point $t=a$ on $x_0$ means there is a solution $y(t)$ of the above Jacobi equation with the integral condition satisfying $y(a)=y(c)=0$ and $y(t)\not\equiv 0$ on $a<t<c$.
 
With a same method in the previous theorem, it can be proved that  $J^{\prime\prime}(x_0,y)>0$ holds if and only if there is no point $c$ conjugate to $a$ on $0<t\neq b$.

\section{\textbf{The length and the Busemann-Hausdorff area in the 2-dimensional Finsler space forms with $k=0$}}

When the 2-dimensional Finsler space form has flag curvature $k=0$, the non-Riemannian metric $F_B$ is given by the Berwald's example on a 2-dimensional open ball $B^2(1)=\{x\in R^2| |x|<1\}$:
\[F_{B}=\frac{(\sqrt{|y|^2-(|x|^2|y|^2-\langle x,y\rangle^2)}+\langle x,y\rangle)^2}{(1-|x|^2)^2\sqrt{|y|^2-(|x|^2|y|^2-\langle x,y\rangle^2)}}\]
where $x\in B^2(1)$ and $y$ is an arbitrary tangent vector in the tangent plane $T_xB^2(1)$.
If using the notation of the spherically symmetric Finsler metrics, this metric can be written as $F_{B}=u\phi(r,s)$, where $\phi=\frac{(\sqrt{1-r^2+s^2}+s)^2}{(1-r^2)^2\sqrt{1-r^2+s^2}}$, $u=|y|$, $r=|x|^2$, $s=\frac{\langle x,y\rangle}{|y|}$. The above 2-dimensional Finsler space form with flag curvature $k=0$ will be denoted by $(F_B, B^2(1))$. 

Suppose a parametrized curve $c(t)=(x_1(t), x_2(t))$ in  the open ball $B^2(1)$  has continuous first and second derivatives for $t\in [t_0, t_1]$, satisfying $c(t_0)=c(t_1)$.
In the Finsler space form $(F_B, B^2(1))$, the length of the curve $c(t)$ is  giving by 
 \begin{eqnarray*}L&=&\int_{t_0}^{t_1}\frac{(\sqrt{(1-x_1^2-x_2^2)(\dot{x_1}^2+\dot{x_2}^2)+(x_1\dot{x_1}+x_2\dot{x_2})^2}+x_1\dot{x_1}+x_2\dot{x_2})^2}{(1-x_1^2-x_2^2)^2\sqrt{(1-x_1^2-x_2^2)(\dot{x_1}^2+\dot{x_2}^2)+(x_1\dot{x_1}+x_2\dot{x_2})^2}}dt\\
&=&\int_{t_0}^{t_1}\frac{(1-x_1^2-x_2^2)(\dot{x_1}^2+\dot{x_2}^2)+2(x_1\dot{x_1}+x_2\dot{x_2})^2}{(1-x_1^2-x_2^2)^2\sqrt{(1-x_1^2-x_2^2)(\dot{x_1}^2+\dot{x_2}^2)+(x_1\dot{x_1}+x_2\dot{x_2})^2}}dt\\
&&+  2\int_{t_0}^{t_1}\frac{x_1\dot{x_1}+x_2\dot{x_2}}{(1-x_1^2-x_2^2)^2}dt\\
&=&\int_{t_0}^{t_1}\frac{(1-x_1^2-x_2^2)(\dot{x_1}^2+\dot{x_2}^2)+2(x_1\dot{x_1}+x_2\dot{x_2})^2}{(1-x_1^2-x_2^2)^2\sqrt{(1-x_1^2-x_2^2)(\dot{x_1}^2+\dot{x_2}^2)+(x_1\dot{x_1}+x_2\dot{x_2})^2}}dt.
\end{eqnarray*}
The last equality holds since the curve is closed and the integral $\int_{t_0}^{t_1}\frac{x_1\dot{x_1}+x_2\dot{x_2}}{(1-x_1^2-x_2^2)^2}dt$ vanishes. 

The Busemann-Hausdorff volume form $\sigma_{BH}$ of a Finsler metric $F$ on a $n$-dimensional manifold $M$ is  defined by 
\[\sigma_{BH}=\frac{Vol(B^n(1))}{Vol\{(y^i)\in R^n| F(x_1,x_2)\le 1\}}.\]  

Similar to the calculation of the Busemann-Hausdorff volume form of the $(\alpha,\beta)$ metric in X. Cheng and Z. Shen's paper \cite{CS} , we have the following lemma for a spherically symmetric Finsler metric.
\begin{lemma}\label{le1}Let $F=u\phi(r,s)$ be a spherically symmetric Finsler metric on an open set $\Omega$ in $R^n$. Then the Busemann-Hausdorff volume form $\sigma_{BH}$ of $F$ is given by
\[\sigma_{BH}(r)=\frac{\int_{0}^{\pi}\sin ^{n-2} t dt}{\int_{0}^{\pi} \frac{\sin^{n-2}t }{\phi(r, r\cos t)^n}dt}.\]
\end{lemma} 
The proof is omitted here and one can consult \cite{CS}  for the details.  

\begin{theorem} \label{th1}
For 2-dimensional non-Riemanian Finsler space form with $k=0$ $(F_{B}, B^2(1))$,  the Busemann-Hausdorff area form can be expressed as
\[\sigma_{BH}(F_{B})=\frac{1}{1+\frac{3}{2}r^2},\]
where $r:=|x|$.
\end{theorem}

\begin{proof}
From Lemma \ref{le1}, we have 
\begin{eqnarray*}
\sigma_{BH}(F_{B})&=&\frac{\int_{0}^{\pi}1dt}{\int_{0}^{\pi}\frac{1}{\phi(r,r\cos t)^2}dt}\\
&=&\frac{\pi}{\int_{0}^{\pi}\frac{(1-r^2)^4(1-r^2\sin^2 t)}{(\sqrt{1-r^2\sin^2 t}+r\cos t)^4}dt}.
\end{eqnarray*}
By a direct calculation, one can obtain
\[\int_{0}^{\pi}\frac{(1-r^2)^4(1-r^2\sin^2 t)}{(\sqrt{1-r^2\sin^2 t}+r\cos t)^4}dt=\pi+\frac{3}{2}\pi r^2.\] 
Therefore $\sigma_{BH}(F_B)=\frac{1}{1+\frac{3}{2}r^2}$.
\end{proof}

The Green formula 
\[\int_{\partial\Omega}Pdx+Qdy=\iint_{\Omega}(\frac{\partial Q}{\partial x}-\frac{\partial P}{\partial y})dxdy\] applies the following Busemann-Hausdorff area formula.
\begin{theorem} For 2-dimensional non-Riemanian Finsler space form with $k=0$ $(F_{B}, B^2(1))$, the Busemann-Hausdorff area $A_{BH}$ enclosed by a simple closed curve $c(t)=(x_1(t),x_2(t))$ where $t\in [t_0,t_1]$ and $c(t_0)=c(t_1)$ is of form
 \[A_{BH}=\frac{1}{3}\int_{t_0}^{t_1}\frac{\ln(3(x_1^2+x_2^2)+2)}{x_1^2+x_2^2}(x_1\dot{x_2}-x_2\dot{x_1})dt.\]
\end{theorem}

\begin{proof}
From Theorem \ref{th1}, the Busemann-Hausdorff area is given by
\[A_{BH}=\iint_{\Omega}\frac{1}{1+\frac{3}{2}(x_1^2+x_2^2)}dx_1dx_2\]
where $\Omega$ is the domain enclosed by a simple closed curve $c(t)=(x(t),y(t))$. It only need to be proved that 
\begin{equation}\label{eq1}
\iint_{\Omega}\frac{1}{1+\frac{3}{2}(x_1^2+x_2^2)}dx_1dx_2=\frac{1}{3}\int_{\partial\Omega}\frac{\ln(3(x_1^2+x_2^2)+2)}{x_1^2+x_2^2}(x_1dx_2-x_2dx_1).\end{equation}
If let $P=-\frac{1}{3}\frac{\ln(3(x_1^2+x_2^2)+2)}{x_1^2+x_2^2}x_2$ and $Q=\frac{1}{3}\frac{\ln(3(x_1^2+x_2^2)+2)}{x_1^2+x_2^2}x_1$, it could be easily checked that
\[\frac{\partial Q}{\partial x_1}-\frac{\partial P}{\partial x_2}=\frac{1}{1+\frac{3}{2}(x_1^2+x_2^2)}\]
holds. By the Green formula, the equation (\ref{eq1}) holds. Thus the formula of the Busemann-Hausdorff area formula of $(F_{B}, B^2(1))$ is yielded. 
\end{proof}

 \section{\textbf{The Euler-Langrange equation of the isoperimetric problem in the 2-dimensional Finsler space forms with $k=0$}}
When using the Busemann-Hausdorff area in $(F_B, B^2(1))$, the isoperimetric problem is stated as to seek 
 \[\max\{A_{BH}(D)|\text{where D is enclosed by an arbitrary simple closed $C^{\infty}$curve}  \}\]
 under the constrained condition the curve's length $L=l$.  This is the classical isoperimetric problem in the theory  of variational calculous.  We can apply the so-called Lagrange multipliers method to handle this problem.
 
Assume a simple closed curve  in  $(F_B, B^2(1))$ is given by $c(t)=(x_1(t),x_2(t))$ where $t\in [t_0,t_1]$ and $c(t_0)=c(t_1)$. 
Let  $A_{BH}=\int_{t_0}^{t_1}f(x_1,x_2,\dot{x_1},\dot{x_2})dt$ be the Busemann-Hausdorff area enclosed by the curve 
and $L=\int_{t_0}^{t_1}g(x_1,x_2,\dot{x_1},\dot{x_2})$
 be the length of the curve where 
\[f(x_1,x_2,\dot{x_1},\dot{x_2})=\frac{1}{3}\frac{\ln(3(x_1^2+x_2^2)+2)}{x_1^2+x_2^2}(x_1\dot{x_2}-x_2\dot{x_1})\]
and
\[g(x_1,x_2,\dot{x_1},\dot{x_2})=\frac{(1-x_1^2-x_2^2)(\dot{x_1}^2+\dot{x_2}^2)+2(x_1\dot{x_1}+x_2\dot{x_2})^2}{(1-x_1^2-x_2^2)^2\sqrt{(1-x_1^2-x_2^2)(\dot{x_1}^2+\dot{x_2}^2)+(x_1\dot{x_1}+x_2\dot{x_2})^2}}.\]
Obviously, $f$ and $g$ are independent of the parameter $t$. What's more, $f$ and $g$ are positively homogeneous in $(\dot{x_1}, \dot{x_2})$ of degree one and have continuous derivatives of the first three orders.

Let
\begin{eqnarray*}
J&=&A_{BH}+\lambda L\\
&=&\int_{t_0}^{t_1}h(x_1,x_2,\dot{x}_1,\dot{x}_2,\lambda) dt,
\end{eqnarray*}
where $h=f+\lambda g$.
The Euler-Lagrange equations of $J$ are
\[\frac{\partial h}{\partial x_1}-\frac{d}{dt}\frac{\partial h}{\partial \dot{x_1}}=0,\quad \frac{\partial h}{\partial x_2}-\frac{d}{dt}\frac{\partial h}{\partial \dot{x_2}}=0.\]
Plugging $h$ into above equation, one can obtain a second ODE system with  many quite long terms. 

Actually, the Euler-Lagrange equations have a concise form when using the polar coordinate.  
In the polar coordinate, let the curve be $c(t)=(r(t)cos(t),r(t)sin(t))$. Then
\begin{equation*}
f(r,\dot{r})=\frac{1}{3}\ln(3r^2+2), \quad g(r,\dot{r})=\frac{r^2(1-r^2)+(1+r^2)\dot{r}^2}{(1-r^2)^2\sqrt{r^2(1-r^2)+\dot{r}^2}}
\end{equation*}
and
\begin{equation}\label{eq2}
h(r,\dot{r})=f+\lambda g=\frac{1}{3}\ln(3r^2+2)+\lambda \frac{r^2(1-r^2)+(1+r^2)\dot{r}^2}{(1-r^2)^2\sqrt{r^2(1-r^2)+\dot{r}^2}}.
\end{equation}
The Euler-Lagrange equation $\frac{\partial h}{\partial r}-\frac{d}{dt}\frac{\partial h}{\partial \dot{r}}=0$ can be simplified as 
\begin{eqnarray} \label{eq3}
&\frac{2r}{3r^2+2}+\frac{\lambda (2r(1-r^2)-2r^3+2r\dot{r}^2)}{(1-r^2)^2D}+\frac{4\lambda (r^2(1-r^2)+(1+r^2)\dot{r}^2)r}{(1-r^2)^3D} \\
&-\frac{1}{2}\frac{\lambda(r^2(1-r^2)+(1+r^2)\dot{r}^2)(2r(1-r^2)-2r^3)}{(1-r^2)^2D^3}-\frac{d}{dt}\frac{\lambda \dot{r}(r^2+r^4-2r^6+(1+r^2)\dot{r}^2)}{(1-r^2)^2D^3}=0, \nonumber
\end{eqnarray}
where $D=\sqrt{r^2(1-r^2)+\dot{r}^2}$. 
Integrating above equation will lead to the following theorem.

\begin{proposition}
If a curve $c(t)=(r(t)cos(t),r(t)sin(t))$ is a solution of the isoperimetric problem in $(F_{B}, B^2(1))$ when using the Busemann-Hausdorff area, then it must satisfy the equation 
\[\frac{\lambda \dot{r}^2(r^2+r^4-2r^6+(1+r^2)\dot{r}^2)}{(1-r^2)^2(r^2(1-r^2)+\dot{r}^2)^{\frac{3}{2}}} -\frac{1}{3}\ln(3r^2+2)-\lambda \frac{r^2(1-r^2)+(1+r^2)\dot{r}^2}{(1-r^2)^2\sqrt{r^2(1-r^2)+\dot{r}^2}}=C_1,\]
where $C_1$ is a constant. 
\end{proposition}
\begin{proof} 
From (\ref{eq2}), we know that $h$ is independent of variable $t$. Thus 
\begin{eqnarray*}
\frac{d}{dt}(\dot{r}\frac{\partial h}{\partial \dot{r}}-h)&=&\dot{r}\frac{d}{dt}(\frac{\partial h}{\partial \dot{r}})-\frac{\partial h}{\partial r}\dot{r}\\
&=&-\dot{r}[\frac{\partial h}{\partial r}-\frac{d}{dt}(\frac{\partial h}{\partial \dot{r}})]\\
&=&0.
\end{eqnarray*}
Integrating this equation will obtain
\[\dot{r}\frac{\partial h}{\partial \dot{r}}-h=C_1.\]
Plugging $h$ into above equation will get the result. 
\end{proof}

It is difficult to get all the solutions of above equation.  However, it is easy to see $r=const$ is one of the solutions. 
\begin{theorem}
The circles $c_0=(a\cos(t), a\sin(t))$ centered at the origin, where $a\in (0,1)$ and $t\in [0, 2\pi]$, are isoperimetric extremal with respect to the integral $J$. Moreover, 
$$\lambda_{0}=-\frac{2(1-a^2)^{\frac{5}{2}}a}{6a^4+7a^2+2}<0.$$
\end{theorem}
\begin{proof}
The curves $c_0=(a\cos(t), a\sin(t))$ are obviously the solutions of the Euler-Lagrange Equations of $J$. Plugging $c_0=(a\cos(t), a\sin(t))$ into the equation (\ref{eq3}) will obtain $$\lambda_{0}=-\frac{2(1-a^2)^{\frac{5}{2}}a}{6a^4+7a^2+2}.$$
For $a \in (0,1)$, $\lambda_{c_0}$ is negative.
\end{proof}

\begin{theorem}
The circles $c_0=(a\cos(t), a\sin(t))$ centered at the origin  are normal, where $a\in (0,1)$ and $t \in [0, 2\pi]$. 
\end{theorem}
\begin{proof}
Along the circles $c_0$, it can be calculated that
  \[P_1=g_{x_1}-\frac{d g_{\dot{x}_1}}{dt}=\frac{(1+2a^2)}{(1-a^2)^{\frac{5}{2}}}\cos(t),\]
  \[P_2=g_{x_2}-\frac{d g_{\dot{x}_2}}{dt}=\frac{(1+2a^2)}{(1-a^2)^{\frac{5}{2}}}\sin(t).\]
Therefore $P_1$ and $P_2$ are not $0$ function and the circles $c_0$ are normal.  
\end{proof}

\section{\textbf{Weierstrass E-function}}
The Weierstrass E-function of the integral $J$ is defined as
\[E(x_1,x_2,\dot{x_1},\dot{x_2},p_1,p_2):=h(x_1,x_2,p_1,p_2)-h(x_1,x_2,\dot{x_1},\dot{x_2})\]
\[-(p_1-\dot{x_1})\frac{\partial h(x_1,x_2,\dot{x}_1,\dot{x}_2)}{\partial \dot{x}_1}-(p_2-\dot{x_2})\frac{\partial h(x_1,x_2,\dot{x}_1,\dot{x}_2)}{\partial \dot{x}_2}\]
and the following negativity of the Weierstrass E-function $E(x_1,x_2,\dot{x_1},\dot{x_2},p_1,p_2)$ can be proved. 

\begin{proposition}
If assuming $\lambda< 0$, then the Weierstrass E-function $E$ of the integral $J$ satisfies 
\[E<0 \]
except for $(p_1,p_2)\neq k (\dot{x}_1, \dot{x}_2) (k>0)$.
\end{proposition}
\begin{proof} 
For our convenance, the following notation are introduced:
$$A=\sqrt{(1-x_1^2-x_2^2)(\dot{x_1}^2+\dot{x_2}^2)+(x_1\dot{x_1}+x_2\dot{x_2})^2},$$
$$B=\sqrt{(1-x_1^2-x_2^2)(p_1^2+p_2^2)+(x_1p_1+x_2p_2)^2},$$
$$C=(1-x_1^2-x_2^2)(\dot{x_1}p_1+\dot{x_2}p_2)+(x_1p_1+x_2p_2)(x_1\dot{x_1}+x_2\dot{x_2}).$$
Notice that the function $h(x_1,x_2,\dot{x_1},\dot{x_2})$ which is of the form
\begin{eqnarray*}
h&=&\frac{1}{3}\frac{\ln (3(x_1^2+x_2^2)+2)}{x_1^2+x_2^2}(x_1\dot{x_2}-x_2\dot{x_1})\\
&&+\lambda\frac{(1-x_1^2-x_2^2)(\dot{x_1}^2+\dot{x_2}^2)+2(x_1\dot{x_1}+x_2\dot{x_2})^2}{(1-x_1^2-x_2^2)^2\sqrt{(1-x_1^2-x_2^2)(\dot{x_1}^2+\dot{x_2}^2)+(x_1\dot{x_1}+x_2\dot{x_2})^2}}
\end{eqnarray*}
can be written as
\begin{eqnarray*}
h&=&\frac{1}{3}\frac{\ln (3(x_1^2+x_2^2)+2)}{x_1^2+x_2^2}(x_1\dot{x_2}-x_2\dot{x_1})\\
&&+\lambda\frac{\sqrt{(1-x_1^2-x_2^2)(\dot{x_1}^2+\dot{x_2}^2)+(x_1\dot{x_1}+x_2\dot{x_2})^2}}{(1-x_1^2-x_2^2)^2}\\
&&+\lambda\frac{(x_1\dot{x_1}+x_2\dot{x_2})^2}{(1-x_1^2-x_2^2)^2\sqrt{(1-x_1^2-x_2^2)(\dot{x_1}^2+\dot{x_2}^2)+(x_1\dot{x_1}+x_2\dot{x_2})^2}}.
\end{eqnarray*}
By using the notation, $h$ is given by
\[
h=\frac{1}{3}\frac{\ln (3(x_1^2+x_2^2)+2)}{x_1^2+x_2^2}(x_1\dot{x_2}-x_2\dot{x_1})+\lambda\frac{A}{(1-x_1^2-x_2^2)^2}+\lambda\frac{(x_1\dot{x_1}+x_2\dot{x_2})^2}{(1-x_1^2-x_2^2)^2A}.
\]
Plugging $h$ into the Weierstrass E-function will yield
\begin{eqnarray*}
E&=&\frac{1}{3}\frac{\ln\big(3(x_1^2+x_2^2)+2\big)}{x_1^2+x_2^2}(x_1p_2-x_2p_1)+\lambda\frac{B}{(1-x_1^2-x_2^2)^2}+\lambda\frac{(x_1p_1+x_2p_2)^2}{(1-x_1^2-x_2^2)^2B}\\
&&-\frac{1}{3}\frac{\ln\big(3(x_1^2+x_2^2)+2\big)}{x_1^2+x_2^2}(x_1\dot{x_2}-x_2\dot{x_1})-\lambda\frac{A}{(1-x_1^2-x_2^2)^2}-\lambda\frac{(x_1\dot{x_1}+x_2\dot{x_2})^2}{(1-x_1^2-x_2^2)^2A}\\
&&-(p_1-\dot{x_1})\big[ -\frac{1}{3}\frac{\ln\big(3(x_1^2+x_2^2)+2\big)}{x_1^2+x_2^2}x_2+\lambda\frac{(1-x_1^2-x_2^2)\dot{x}_1+(x_1\dot{x}_1+x_2\dot{x}_2)x_1}{(1-x_1^2-x_2^2)^2A}   \\
&&+\lambda\frac{2(x_1\dot{x}_1+x_2\dot{x}_2)x_1}{(1-x_1^2-x_2^2)^2A}-\lambda\frac{(x_1\dot{x}_1+x_2\dot{x}_2)^2\big((1-x_1^2-x_2^2)\dot{x}_1+(x_1\dot{x}_1+x_2\dot{x}_2)x_1\big)}{(1-x_1^2-x_2^2)^2A^3}\big]   \\
&&-(p_2-\dot{x_2})\big[ \frac{1}{3}\frac{\ln\big(3(x_1^2+x_2^2)+2\big)}{x_1^2+x_2^2}x_1+\lambda\frac{(1-x_1^2-x_2^2)\dot{x}_2+(x_1\dot{x}_1+x_2\dot{x}_2)x_2}{(1-x_1^2-x_2^2)^2A}  \\
&&+\lambda\frac{2(x_1\dot{x}_1+x_2\dot{x}_2)x_2}{(1-x_1^2-x_2^2)^2A}-\lambda\frac{(x_1\dot{x}_1+x_2\dot{x}_2)^2\big((1-x_1^2-x_2^2)\dot{x}_2+(x_1\dot{x}_1+x_2\dot{x}_2)x_2\big)}{(1-x_1^2-x_2^2)^2A^3}\big].
\end{eqnarray*}
Furthermore, $E$ can be simplified to 
\begin{eqnarray*}
E&=&-\lambda\frac{A}{(1-x_1^2-x_2^2)^2}+\lambda\frac{B}{(1-x_1^2-x_2^2)^2}-\lambda\frac{(x_1\dot{x_1}+x_2\dot{x_2})^2}{(1-x_1^2-x_2^2)^2A}+\lambda\frac{(x_1p_1+x_2p_2)^2}{(1-x_1^2-x_2^2)^2B}\\
&&-\lambda(p_1-\dot{x_1})\frac{(1-x_1^2-x_2^2)\dot{x}_1+(x_1\dot{x}_1+x_2\dot{x}_2)x_1}{(1-x_1^2-x_2^2)^2A}\\
&&-\lambda(p_2-\dot{x_2})\frac{(1-x_1^2-x_2^2)\dot{x}_2+(x_1\dot{x}_1+x_2\dot{x}_2)x_2}{(1-x_1^2-x_2^2)^2A}\\
&&-\lambda(p_1-\dot{x_1})\frac{2(x_1\dot{x}_1+x_2\dot{x}_2)x_1}{(1-x_1^2-x_2^2)^2A}-\lambda(p_2-\dot{x_2})\frac{2(x_1\dot{x}_1+x_2\dot{x}_2)x_2}{(1-x_1^2-x_2^2)^2A}\\
&&+\lambda\frac{(p_1-\dot{x_1})[(1-x_1^2-x_2^2)\dot{x}_1+(x_1\dot{x}_1+x_2\dot{x}_2)x_1](x_1\dot{x}_1+x_2\dot{x}_2)^2}{(1-x_1^2-x_2^2)^2A^3}\\
&&+\lambda\frac{(p_2-\dot{x_2})[(1-x_1^2-x_2^2)\dot{x}_2+(x_1\dot{x}_1+x_2\dot{x}_2)x_2](x_1\dot{x}_1+x_2\dot{x}_2)^2}{(1-x_1^2-x_2^2)^2A^3}\\
&=&-\lambda\frac{A}{(1-x_1^2-x_2^2)^2}+\lambda\frac{B}{(1-x_1^2-x_2^2)^2}-\lambda\frac{(x_1\dot{x_1}+x_2\dot{x_2})^2}{(1-x_1^2-x_2^2)^2A}+\lambda\frac{(x_1p_1+x_2p_2)^2}{(1-x_1^2-x_2^2)^2B}\\
&&+\lambda \frac{1}{(1-x_1^2-x_2^2)^2A}(A^2-C)+\lambda\frac{2(x_1\dot{x}_1+x_2\dot{x}_2)(x_1\dot{x_1}+x_2\dot{x_2}-x_1p_1-x_2p_2)}{(1-x_1^2-x_2^2)^2A}\\
&&+\lambda\frac{(x_1\dot{x}_1+x_2\dot{x}_2)^2}{(1-x_1^2-x_2^2)^2A^3}(C-A^2)\\
&=&\lambda\frac{B}{(1-x_1^2-x_2^2)^2}+\lambda\frac{(x_1p_1+x_2p_2)^2}{(1-x_1^2-x_2^2)^2B}-\lambda\frac{C}{(1-x_1^2-x_2^2)^2A}\\
&&-\frac{2\lambda}{(1-x_1^2-x_2^2)^2A}(x_1p_1+x_2p_2)(x_1\dot{x_1}+x_2\dot{x_2})+\lambda\frac{(x_1\dot{x}_1+x_2\dot{x}_2)^2C}{(1-x_1^2-x_2^2)^2A^3}\\
&=&-\frac{\lambda}{(1-x_1^2-x_2^2)^2A^3B}\big(-A^3B^2-(x_1p_1+x_2p_2)^2A^3\\
&&+2(x_1p_1+x_2p_2)(x_1\dot{x_1}+x_2\dot{x_2})A^2B+CB(1-x_1^2-x_2^2)(\dot{x}_1^2+\dot{x}_2^2)\big)\\
&\leq&-\frac{\lambda}{(1-x_1^2-x_2^2)^2A^3B}\big(-A^3B^2-(x_1p_1+x_2p_2)^2A^3+(x_1\dot{x}_1+x_2\dot{x}_2)^2AB^2\\
&&+(x_1p_1+x_2p_2)^2A^3+CB(1-x_1^2-x_2^2)(\dot{x}_1^2+\dot{x}_2^2)\big)\\
&\leq&-\frac{\lambda}{(1-x_1^2-x_2^2)^2A^3B}\big(-A^3B^2+(x_1\dot{x}_1+x_2\dot{x}_2)^2AB^2\\
&&+AB^2(1-x_1^2-x_2^2)(\dot{x}_1^2+\dot{x}_2^2)\big)=0.
\end{eqnarray*}
The last inequality holds for the Cauchy inequality $C\leq AB$. Obviously, $E=0$ holds if and only if $(p_1, p_2)=k(\dot{x}_1,\dot{x}_2) (k>0)$.
\end{proof}
It is easy to see that the proposition implies the following theorem.
\begin{theorem}
Let $c_0$ be the circle centered at the origin in $B^2(1)$. For each $(x,\dot{x})$ in a neighborhood of $c_0$, the Weierstrass function
$$E(x,\dot{x},u)<0$$ 
holds for every admissible set $(x,u)\neq (x,k\dot{x}) (k>0)$.

\end{theorem}

\section{\textbf{The conjugate points of the critical circles}}
For the isoperimetric extremal circles $c_0=(a\cos(t), a\sin(t))$ where $a\in (0, 1)$ in $(F_{B}, B^2(1))$,   the integrands of the Busemann-Hausdorff area $A$ and the length $L$ are given by
\[f(c_0)=\frac{1}{3}\ln (3a^2+2), \quad g(c_0)=\frac{a}{(1-a^2)^{\frac{3}{2}}}.\]
Therefore $$h(c_0)=f(c_0)+\lambda g(c_0)=\frac{1}{3}\ln (3a^2+2)+\lambda\frac{a}{(1-a^2)^{\frac{3}{2}}}.$$
The Jacobi equation along $c_0$ is of the form
\[\Psi(w)+\mu U=0\]
where $\Psi(w)=H_2w-\frac{d}{dt}(H_1w^\prime)$, $U=g_{x_1\dot{x}_2}-g_{\dot{x}_1x_2}+g_1(\dot{x}_1\ddot{x}_2-\ddot{x}_1\dot{x}_2)$,  $H_1=\frac{h_{\dot{x}_1\dot{x}_1}}{\dot{x}_2^2}$, 
$H_2=\frac{h_{x_1x_1}-\ddot{x}_2^2H_1-\frac{d J}{dt}}{\dot{x}_2^2}$, $J=h_{x_1\dot{x}_1}-\dot{x}_2\ddot{x}_2H_1$. Furthermore, $w$ satisfies 
\[\int_{t_0}^t Uw dt=0.\]
Plugging $g$, $h$ and $\lambda$ into above Jacobi equation will yield
\[\frac{2}{a^2(3a^2+2)}\frac{d^2w}{dt^2}+\frac{2(12a^6+46a^4+19a^2-2)}{a^2(3a^2+2)^2(2a^2+1)(a^2-1)}w+\mu\frac{2a^2+1}{a(1-a^2)^2\sqrt{1-a^2}}=0.\]
It can be written as 
\[\frac{d^2w}{dt^2}+\frac{12a^6+46a^4+19a^2-2}{(3a^2+2)(2a^2+1)(a^2-1)}w+\mu\frac{a(2a^2+1)(3a^2+2)}{2(1-a^2)^2\sqrt{1-a^2}}=0.\]
Notice that 
\[
\frac{12a^6+46a^4+19a^2-2}{(3a^2+2)(2a^2+1)(a^2-1)}: \left\{ \begin{array}{ll}  >0, &a\in (0, a_0),\\
=0, &a=a_0,\\
<0, &a\in (a_0,1),
 \end{array}\right.
\]
where $a_0$ is the solution of the equation $12a^6+46a^4+19a^2-2=0$.
Thus the solution of the Jacobi equation is 
\[w=c_1\theta_1(t)+c_2\theta_2(t)+\mu \theta_3(t) \]
where $c_1$ and $c_2$ are arbitrary constants and 
\[\theta_1(t)=\left\{ \begin{array}{ll}  \sin\sqrt{\frac{12a^6+46a^4+19a^2-2}{(3a^2+2)(2a^2+1)(a^2-1)}}t, & a\in (0, a_0),\\
t, &a=a_0,\\
\sinh \sqrt{-\frac{12a^6+46a^4+19a^2-2}{(3a^2+2)(2a^2+1)(a^2-1)}}t, & a\in (a_0, 1),
 \end{array}\right.
\]

\[\theta_2(t)=\left\{ \begin{array}{ll}\cos\sqrt{\frac{12a^6+46a^4+19a^2-2}{(3a^2+2)(2a^2+1)(a^2-1)}}t, & a\in (0, a_0),\\
1, &a=a_0,\\
\cosh \sqrt{-\frac{12a^6+46a^4+19a^2-2}{(3a^2+2)(2a^2+1)(a^2-1)}}t, & a\in (a_0, 1), 
\end{array}\right.
\]

\[\theta_3(t)=\left\{ \begin{array}{ll}\frac{a(2a^2+1)^2(3a^2+2)^2}{2(1-a^2)\sqrt{1-a^2}(12a^6+46a^4+19a^2-2)}, & a\in (0,a_0)\cup (a_0,1),\\
-\frac{a(2a^2+1)(3a^2+2)}{4(1-a^2)^2\sqrt{1-a^2}}t^2, & a=a_0.
\end{array}\right.
\]

Along the isoperimetric extremal circles $c_0$, $t_1$ is the conjugate to the point $t_0$ if and only if  it is possible to find  the constants $c_1$, $c_2$, $\mu$ such that
\[w(t_0)=c_1\theta_1(t_0)+c_2\theta_2(t_0)+\mu\theta_3(t_0)=0,\]
\[w(t_1)=c_1\theta_1(t_1)+c_2\theta_2(t_1)+\mu\theta_3(t_1)=0,\]
\[\int_{t_0}^{t_1}Uw dt=c_1\int_{t_0}^{t_1}U\theta_1dt+c_2\int_{t_0}^{t_1}U\theta_2dt+\mu\int_{t_0}^{t_1}U\theta_3dt=0.\]
Therefore it is necessary that
\[D(t_0, t_1)=\left| \begin{array}{lll}
\theta_1(t_0) &\theta_2(t_0)  &\theta_3(t_0)\\
\theta_1(t_1) &\theta_2(t_1)  &\theta_3(t_1)\\
\int_{t_0}^{t_1}U\theta_1dt &\int_{t_0}^{t_1}U\theta_2dt &\int_{t_0}^{t_1}U\theta_3dt
\end{array}\right|=0.\]
On the other hand, it is easy to get the following expression of $D(t_0, t_1)$:
\[
D(t_0, t_1)=\left\{ \begin{array}{ll}
4\frac{b_3}{b_1}U\sin \frac{b_1(t_1-t_0)}{2}[ \sin\frac{b_1(t_1-t_0)}{2}-\frac{b_1(t_1-t_0)}{2} \cos\frac{b_1(t_1-t_0)}{2}], & a\in (0, a_0),\\
\frac{1}{6}b_4U(t_1-t_0)^2& a=a_0,\\
4\frac{b_3}{b_2}U\sinh \frac{b_2(t_1-t_0)}{2}[ \sinh\frac{b_2(t_1-t_0)}{2}-\frac{b_2(t_1-t_0)}{2} \cos\frac{b_2(t_1-t_0)}{2}], & a\in (a_0, 1),
\end{array}\right.\]
where \[b_1=\sqrt{\frac{12a^6+46a^4+19a^2-2}{(3a^2+2)(2a^2+1)(a^2-1)}}, \quad b_2=\sqrt{-\frac{12a^6+46a^4+19a^2-2}{(3a^2+2)(2a^2+1)(a^2-1)}},\] 
\[b_3=\frac{a(2a^2+1)^2(3a^2+2)^2}{2(1-a^2)\sqrt{1-a^2}(12a^6+46a^4+19a^2-2)}, \quad b_4=-\frac{a(2a^2+1)(3a^2+2)}{4(1-a^2)^2\sqrt{1-a^2}},\]
\[U=\frac{2a^2+1}{a(1-a^2)^2\sqrt{1-a^2}}.\]

When $a\in (0, a_0)$, it can be concluded that the conjugate point $t_1$ is  given by
\[t_1=t_0+\frac{2\pi}{b_1}.\]
From $b_1<1$, we infer there are no conjugate points along $c_0$ where $t\in [0, 2\pi)$.

When $a\in [a_0, 1)$, $D(t_0, t_1)>0$ if $t_1>t_0$ implies that there are no conjugate points along $c_0$.  Hence we have the following theorem.

\begin{theorem}
Along the isoperimetc extremal circles $c_0=(a\cos(t), a\sin(t))$ of the integral $J$ in $(F_{B}, B^2(1))$, there are no conjugate points.
\end{theorem}

\section{\textbf{The conclusion and a conjecture}}

\begin{theorem}
Let $c_0$ be the circle centered at the origin in $B^2(1)$, then $c_0$ is a proper strong maximum of the isoperimetric problem. 
\end{theorem}

\begin{proof} According to the Theorem 2.3  and the previous results,  there is only one sufficient condition left to be verified: 
  $$\sum_{i,j=1}^2h_{\dot{x}_i\dot{x}_j}y_iy_j<0$$
holds along $c_0$ for all $y\neq k\dot{c}_0(t).$ 
Plugging $c_0=(a\cos(t), a\sin(t))$ into $\sum_{i,j=1}^2h_{\dot{x}_i\dot{x}_j}y_iy_j$ will obtain
\[\sum_{i,j=1}^2h_{\dot{x}_i\dot{x}_j}y_iy_j=\frac{\lambda_0(2a^2+1)}{a(1-a^2)^{\frac{5}{2}}}\big(\cos (t)y_1+\sin(t)y_2\big)^2.\]
Since $\lambda_0<0$, the inequality  $\sum_{i,j=1}^2h_{\dot{x}_i\dot{x}_j}y_iy_j<0$ hods for all $y\neq k\dot{c}_0(t)$. 
\end{proof}

Furthermore, we conjecture that $c_0$ achieves the global maximum of the isoperimetric problem in $(F_B, B^2(1))$, when using the Busemann-Hausdorff area. 
\begin{conjecture}
Let $c_0$ be the circle centered at the origin in $(F_B, B^2(1))$, then $c_0$ encloses the maximal Busemann-Hausdorff area among all the simple closed curves which are smooth and have the fixed length. 
\end{conjecture}


{\small DEPARTMENT OF MATHEMATICS, EAST CHINA NORMAL UNIVERSITY }

{\small SHANGHAI 200062, CHINA}\newline
{\small E-mail address: lfzhou@math.ecnu.edu.cn}
\end{document}